\documentclass[a4paper,UKenglish, dvipsnames]{article}

\usepackage{authblk} 

\usepackage[T1]{fontenc}
\usepackage{lmodern}

\date{}

\usepackage[backend=biber,
            natbib=true, %
            sortcites = true,
            style=numeric,
            url=false,
            doi=true,
            giveninits=true,
            maxbibnames=99]{biblatex}

\DeclareSourcemap{
  \maps[datatype=bibtex]{
    \map{
      \step[fieldset=issn, null]
      \step[fieldset=editor, null]
      \step[fieldset=address, null]
    }
  }
}

\usepackage{ifdraft}
\usepackage[colorinlistoftodos,obeyFinal]{todonotes}
\presetkeys{todonotes}{inline, color=topaz!40, bordercolor=topaz,caption={}}{}

\definecolor{fresh}{HTML}{1e5e06}
\definecolor{checked}{HTML}{02087F}
\definecolor{final}{HTML}{A66911}
\definecolor{wrong}{HTML}{c90000}
\definecolor{skip}{HTML}{ffffff}
\definecolor{normal}{HTML}{000000}

\definecolor{ruby}{RGB}{190,0,1}
\definecolor{emerald}{RGB}{26,121,42}
\definecolor{topaz}{RGB}{236,185,57}
\definecolor{sapphire}{RGB}{14,26,164}
\definecolor{blue}{RGB}{14,26,164}

\usepackage{amsmath, amsthm, amssymb, calc} %
\usepackage{mathrsfs}
\usepackage{thm-restate}
\usepackage{mathtools}
\usepackage{url}
\usepackage{mleftright}
\usepackage{microtype}
\usepackage[small,basic]{complexity}
\usepackage[normalem]{ulem}

\usepackage[hidelinks]{hyperref}

\usepackage{enumitem}

\usepackage[nameinlink,capitalize]{cleveref}
\crefname{condition}{Condition}{Conditions}

\newcommand{\appref}[1]{\hyperref[#1]{Appendix~\ref*{#1}}}

\newtheorem{theorem}{Theorem}
\newtheorem{proposition}[theorem]{Proposition}
\newtheorem{lemma}[theorem]{Lemma}
\newtheorem{corollary}[theorem]{Corollary}
\newtheorem{conjecture}[theorem]{Conjecture}

\theoremstyle{definition}
\newtheorem{example}[theorem]{Example}
\newtheorem{remark}[theorem]{Remark}

\newtheorem*{ndefinition}{Definition}

\renewcommand{\epsilon}{\ensuremath\varepsilon}

\renewcommand{\Lambda}{\ensuremath\mu}

\renewcommand{\R}{\mathbb{R}}
\newcommand{\Z}{\mathbb{Z}}
\newcommand{\N}{\mathbb{N}}

\newcommand{\Q}{\mathbb{Q}}

\newcommand{\CC}{\mathbb{C}}

\newcommand{\FF}{\mathbb{F}}

\newcommand{\NN}{\mathbb{N}}

\newcommand{\QQ}{\mathbb{Q}}

\newcommand{\ZZ}{\mathbb{Z}}

\newcommand{\cD}{\mathcal{D}}

\renewcommand{\a}{\alpha}
\renewcommand{\b}{\beta}

\newcommand{\e}{\varepsilon}

\newcommand{\set}[2]{\left\{ #1 \ \middle| \ #2 \right\}}

\newcommand{\bra}[1]{\mleft( #1 \mright)}

\newcommand{\abs}[1]{\left|#1\right|}

\newcommand{\rep}{\mathrm{rep}_p}

\makeatletter
\DeclareRobustCommand\bigop[1]{%
  \mathop{\vphantom{\sum}\mathpalette\bigop@{#1}}\slimits@
}
\newcommand{\bigop@}[2]{%
  \vcenter{%
    \sbox\z@{$#1\sum$}%
    \hbox{\resizebox{\ifx#1\displaystyle.9\fi\dimexpr\ht\z@+\dp\z@}{!}{$\m@th#2$}}%
  }%
}
\makeatother

\newcommand{\Eop}{\DOTSB\bigop{\mathbb{E}}}

\newcommand*\pFq[2]{{}_{#1}F_{#2}}

\DeclareMathOperator{\Gal}{Gal}

\DeclareMathOperator{\weil}{Weil}

\newcommand*\wc{{}\cdot{}}

\providecommand*{\eu}%
{\ensuremath{\mathrm{e}}}
\providecommand*{\iu}%
{\ensuremath{\mathrm{i}}}

\newcommand{\seq}[3][{}]{\langle #2 \rangle_{#3}^{#1}}

\title{On the growth of hypergeometric sequences}

\author[1]{George Kenison}%
\affil[1]{{Liverpool John Moores University}}

\author[2]{Jakub Konieczny}
\affil[2]{{University of Oxford}}

\author[2,3,4]{Florian Luca}
\affil[3]{{Stellenbosch University}}
\affil[4]{{Max Planck Institute for Software Systems}}

\author[5]{Andrew Scoones}
\affil[5]{{University Paris-Est Créteil}}

\author[6]{Mahsa Shirmohammadi}
\affil[6]{{CNRS, IRIF}}

\author[2]{James Worrell}

\begin{document}

\maketitle

\begin{abstract}
Hypergeometric sequences obey first-order linear recurrence relations with polynomial coefficients and are commonplace throughout the mathematical and computational sciences.
For certain classes of hypergeometric sequences, we prove linear growth estimates on their Weil heights.
We give an application of our effective results, towards the Membership Problem from Computer Science.  Recall that Membership asks to procedurally determine whether a specificed target is an element of a given recurrence sequence. 
\end{abstract}

\section{Introduction}\label{sec:intro}

In this work, we estimate the growths of Weil complexity of hypergeometric sequences.
Recall that a rational-valued sequence is \emph{hypergeometric} if its terms obey a first-order recurrence relation with polynomial coefficients.
Specifically, a rational-valued sequence \(\seq[\infty]{u_n}{n=0}\) is hypergeometric if its terms obey a recurrence relation of the form
\begin{equation}\label{eq:setup:def-u}
	f(n)u_{n} = g(n)u_{n-1},
\end{equation}
where $f(x),g(x) \in \QQ[x]$ are polynomials with rational coefficients and the initial value $u_0 \in \QQ$ is rational.
Here and throughout we shall assume that \(f(x)\) has no non-negative integer zeroes.
This setup and the assumption on \(f(x)\) means that the recurrence relation \eqref{eq:setup:def-u} uniquely defines an infinite sequence of rational numbers.

Arguably, the hypergeometric sequences constitute the simplest class of P-finite sequences.
Recall that a sequence is \emph{P-finite} (sometimes \emph{holonomic})
if its terms obey a linear recurrence relation with polynomial coefficients.
Hypergeometric sequences appear throughout the mathematical and computational sciences in relation to their generating functions. Indeed, these generating functions encompass the common trigonometric and hypergeometric functions and have numerous applications in analytic combinatorics and algebraic computation~\cite{flajolet2009analytic, kauers2011tetrahedron}.

\subsection*{Main Contributions}
 Given a hypergeometric sequence \(\seq[\infty]{u_n}{n=0}\) that obeys recurrence relation \eqref{eq:setup:def-u}, we call the roots of the polynomial coefficients \(f\) and \(g\) the \emph{parameters} of sequence \(\seq[\infty]{u_n}{n=0}\).
 We define classes \hyperlink{defin:C}{\(\mathscr{C}\)} and \hyperlink{defin:D}{\(\mathscr{D}\)} of hypergeometric sequences that make additional assumptions on the parameters as laid out below.
 Our main contributions are \cref{thm:main,thm:asymmetry}, which give linear growth estimates on the Weil heights of hypergeometric sequences in these classes.
 Our results generalise the following observation.
 When \(\seq[\infty]{u_n}{n=0}\) is a non-constant geometric sequence of the form \(u_n = \alpha^n\), we have that \(h_{\weil}(u_n) = n h_{\weil}(\alpha)\).
 We note that the growths rates of rational hypergeometric sequences (i.e., the class of \(\seq[\infty]{u_n}{n=0}\) for which \(u_n = q(n)\) with \(q\in\Z[n]\)) are given by the well-known estimate \(h_{\weil}(u_n) = \deg(q) \log(n) + O(1)\) (see \cref{prop:heights}).
 Thus we exclude the class of rational hypergeometric sequences in \hyperlink{defin:C}{\(\mathscr{C}\)} and \hyperlink{defin:D}{\(\mathscr{D}\)} below.

 Henceforth for functions \(a,b\colon \N_0 \to \R\), we shall employ the standard \emph{Vinogradov} and \emph{big-O} notations \(a(n)\gg b(n)\) and \(a(n) = O(b(n))\) to indicate that there exist constants \(N\in\N_0\) and \(C>0\) such that for all \(n\ge N_0\), \(|a(n)|\ge C |b(n)|\) and \(|a(n)|\le C|b(n)|\), respectively.

 \begin{ndefinition}[Class \(\mathscr{C}\)]
 \hypertarget{defin:C}
     Let \(\mathscr{C}\) denote the family of {non-rational} hypergeometric sequences whose parameters lie in $\QQ(\sqrt{\Delta_1}) \cup \QQ(\sqrt{\Delta_2}) \cup \QQ(\sqrt{\Delta_1\Delta_2})$ for some square-free $\Delta_1,\Delta_2, \Delta_1\Delta_2 \in \ZZ$.
 \end{ndefinition}

\begin{restatable}{theorem}{theoremmain}
\label{thm:main}
For hypergeometric sequences \(\seq[\infty]{u_n}{n=0}\) in \hyperlink{defin:C}{class~\(\mathscr{C}\)}, %
we have \(h_{\weil}(u_n) \gg n\).  Here the implied constant depends only on \(\seq[\infty]{u_n}{n=0}\).
\end{restatable}
The proof of \cref{thm:main} is a straightforward corollary of \cref{prop:quadratic,prop:heights} in \cref{sec:growth}.

 \begin{ndefinition}[Class \(\mathscr{D}\)] 
 \hypertarget{defin:D}
     Let \(\mathscr{D}\) denote the family of {non-rational} hypergeometric sequences
     whose parameters
     \(\alpha_1,\ldots, \alpha_d\) and \(\beta_1,\ldots, \beta_d\)
     (the roots of \(f\) and \(g\) respectively) satisfy the following condition: %
     there is no permutation \(\sigma\in S_d\) for which
    \(\mathbb{Q}(\alpha_i) = \mathbb{Q}(\beta_{\sigma (i)})\) holds for all \(1\le i\le d\).
 \end{ndefinition}

  \begin{restatable}{theorem}{theoremasymmetry}
 \label{thm:asymmetry}
 Let \(\seq[\infty]{u_n}{n=0}\) be a hypergeometric sequence in \hyperlink{defin:D}{class \(\mathscr{D}\)} that, in addition, satisfies either of the following conditions.
\begin{enumerate}
    \item  Each of the irreducible factors of the polynomial \(fg\) has degree at most two.
    \item The splitting field of \(fg\) is cyclotomic.
\end{enumerate}
Then \(h_{\weil} \gg n\).  Here the implied constant is computable and, further, depends only on \(fg\) and a prime \(p\).
 \end{restatable}

\cref{thm:asymmetry} follows as a straightforward corollary of \cref{prop:intermediate,prop:cyclic} in \cref{sec:divergence}.
{A minor contribution in \cref{sec:membership} is an application of our effective result in \cref{thm:asymmetry} towards decision procedures for the Membership Problem from theoretical computer science (\cref{thm:intermediate}). 
 }

\subsection*{Approach}
On the one hand, \cref{thm:main,thm:asymmetry} both concern linear growth estimates for the Weil heights of hypergeometric sequences.
On the other hand, the approaches taken towards these theorems are fundamentally different.
We arrive at the non-effective result in \cref{thm:main} (\cref{sec:growth}) by way of a lower-bound on the number of large prime divisors that contribute towards the linear growth of the Weil heights. %
By contrast, the effective results in \cref{thm:asymmetry} (\cref{sec:divergence}) follow from observations on the \(p\)-adic valuations of certain hypergeometric sequences. More specifically, we prove the existence of a prime \(p\) for which the \(p\)-adic valuations of such sequences diverge. In fact, the set of primes for which this estimate holds has positive density by Chebotarev's theorem (see \cref{prop:large:Galois}); however, for our purposes we need only exhibit one such prime.

\subsection*{Related Works}

\subsubsection*{The prime divisors of hypergeometric sequences}
The \(p\)-adic techniques herein bear many similarities with the methods employed in previous works on hypergeometric sequences.
Researchers have long been interested in developing criteria to establish whether the terms of a hypergeometric sequence are integer valued.
Research in this direction includes early work by Landau~\cite{landau1900factorielles}, which uses \(p\)-adic analysis to establish a necessary and sufficient condition for the integrality in the class of factorial hypergeometric sequences. Authors such as Dwork~\cite{dwork1973hypergeometric} and Christol~\cite{christol1986fonctions} gave criteria for the p-adic integrality of hypergeometric sequences with rational parameters.
Closer to our setting, Hong and Wang~\cite{hongarxiv2016} establish a
criterion for the integrality of hypergeometric series with parameters from quadratic
fields.

More recent work, by Franc, Gannon, and Mason~\cite{franc2018unbounded}, considers \(p\)-adic unboundedness (therein a hypergeometric sequence is \emph{\(p\)-adically unbounded} if arbitrarily high powers of \(p\) appear in the denominators of the terms of the sequence).
Those authors show that the set of primes where the coefficient sequence of a hypergeometric series \(\pFq{2}{1}\) with rational parameters is unbounded is (essentially) given by a finite union certain arithmetic progressions. 
The main result in \cite{franc2020densities} gave a formulation for the Dirichlet density of the set of \(p\)-adically bounded primes for such hypergeometric sequences.

Previous works~\cite{maynard2021lower,moll2009,kenison2023membership} have leveraged techniques concerning prime divisors in order to characterise the asymptotic growth of \(\nu_p(u_n)\) as \(n\to\infty\) where \(\seq[\infty]{u_n}{n=0}\) is monic hypergeometric sequence.
Recall that a hypergeometric sequence \(\seq[\infty]{u_n}{n=0}\) is \emph{monic} if it satisfies a first-order recurrence relation of the form \(u_n = g(n) u_{n-1}\).
The characterisations for asymptotic growth are given in terms of the number of roots of \(g\) in \(\Z/p\Z\), which we obtain from Hensel's lemma.

Our result in \cref{thm:main} establishes a growth estimate for hypergeometric sequences with quadratic parameters.
Our approach
relies on machinery developed in the study of roots of quadratic congruences to prime moduli.
Groundbreaking work by Duke, Friedlander, and Iwaniec~\cite{dukefriedlanderiwaniec1995} showed that, in the limit, the normalised roots of a quadratic polynomial with negative discriminant are uniformly distributed for prime moduli. 
In this work, we employ a refined version of this result (\cref{thm:toth} due to T\'{o}th) that establishes uniform distribution for prime moduli in an infinite arithmetic progression.

\subsubsection*{Membership for hypergeometric sequences}
In \cref{sec:membership}, we give an application of our effective results in \cref{sec:divergence} towards the Membership Problem.
Recall that Membership asks to procedurally determine whether a chosen target value is an element of a given sequence.
We postpone our discussion of the background and motivation for the Membership Problem to \cref{sec:membership}.

Our approach towards growth estimates via prime divisibility properties of hypergeometric sequences is reminiscent of the approaches in two previous works \cite{kenison2023membership, nosan2022membership} on the Membership Problem for hypergeometric sequences.
In \cite{nosan2022membership}, the authors established decidability of the Membership Problem for hypergeometric sequences with rational parameters.
Closer to our setting, the authors of \cite{kenison2023membership} proved that the Membership Problem is decidable for the class of sequences whose polynomial coefficients are both monic and split over a quadratic field (in other words, the parameters of the sequences are integers in a quadratic extension of \(\Q\)).
For comparison, our non-effective growth estimate in \cref{thm:main} does not assume that \(f\) and \(g\) are monic and, further, relaxes the condition that the parameters are elements of a single quadratic number field to that of elements of a union of quadratic fields (see class \hyperlink{defin:C}{\(\mathscr{C}\)}).

An entirely different approach towards Membership for hypergeometric sequences is seen in~\cite{kenison2024threshold}.
Therein the (un)conditional decidability results properties on the algebraic independence between mathematical constant such as \(\pi\), \(\eu\), and \(\eu^{\pi}\). 
The conditional decidability results for Membership Problem for hypergeometric sequences with quadratic parameters in \cite{kenison2024threshold} are subject to the truth of a weak form of Schanuel's conjecture~\cite{lang1966introduction}.

\subsubsection*{Growth estimates for C-finite sequences}
We step back from the class of hypergeometric sequences and briefly consider growth estimates for the class of C-finite sequences. 
Recall that an integer-valued sequence \(\seq[\infty]{u_n}{n=0}\) is \emph{C-finite} if it obeys a linear recurrence relation of the form \(u_{n+d} = a_{d-1} u_{n+d-1} + \cdots + a_1 u_{n+1} + a_0 u_n\) where \(a_0,a_1,\ldots, a_{d-1}\in\Z\) and \(a_0\neq 0\).
Thus a given C-finite sequence is uniquely defined by its recurrence relation and a given set of initial values \(u_0,\ldots, u_{d-1}\).
The polynomial \(f(x) = x^d - a_{d-1}x^{d-1} - \cdots - a_1x - a_0\) and its roots \(\lambda_1,\ldots, \lambda_m\) are the \emph{characteristic polynomial} and \emph{characteristic roots} associated with the relation.
Such a sequence is \emph{non-degenerate} if none of the characteristic roots nor ratios of distinct characteristic roots is a root of unity.  
If \(\seq[\infty]{u_n}{n=0}\) is degenerate, then there exists a computable constant \(M\) such that each subsequence \(\seq[\infty]{u_{nM+r}}{n=0}\) with \(r\in\{0,1,\ldots, M-1\}\) is non-degenerate.

Let \(\seq[\infty]{u_n}{n=0}\) be an integer linear recurrence sequence and \(r,\alpha>0\) respectively denote the maximum modulus and maximum multiplicity of its characteristic roots, then standard observations show that \(u_n = O(n^\alpha r^n)\) where the implied constant is effectively computable (cf.~\cite{everest2003recurrence}).
Loxton and van der Poorten~\cite{loxton1977growth} predicted that non-degenerate integer-valued C-finite sequences attain the maximal possible growth rate; that is to say, for each \(\varepsilon>0\) there is an effectively computable constant \(C(\varepsilon)\) such that \(|u_n|>r^{n(1-\varepsilon)}\) whenever \(n>C(\varepsilon)\).
Employing techniques on the sums of $S$-units due to Evertse~\cite{evertse1984sums}, independent works by Fuchs and Heintze~\cite[Theorem A.1]{fuchs2021} and Karimov et al.~\cite[Theorem 2]{karimov2023power} have given non-effective proofs of this conjecture.
In related work, Noubissie~\cite{noubissie2025} has made recent progress in the direction of the conjecture by giving explicit upper bounds on the number solutions of \(|u_n|<r^{n(1-\varepsilon)}\).

\section{Preliminaries}\label{sec:preliminaries}

\subsection*{Hypergeometric sequences}

Let $\seq[\infty]{u_n}{n=0}$ be a hypergeometric sequence satisfying the recurrence relation \begin{equation}%
	f(n)u_{n} = g(n)u_{n-1}, \tag{\ref{eq:setup:def-u}}
\end{equation}
for all \(n \geq 1\), where $f,g\in \QQ[x]$ are polynomials with rational coefficients and the initial value $u_0 \in \QQ$ is rational. 
We make the additional assumption that the coefficient $f$ in \eqref{eq:setup:def-u} has no positive integer roots, which ensures that the terms $u_n \in\QQ$ are well-defined for each $n \geq 1$. 
To avoid triviality, we also assume that $g$ has no positive integer roots, since otherwise $u_n = 0$ for all sufficiently large $n$. 
Thus, letting $r(x) = g(x)/f(x) \in \QQ(x)$ denote the ratio between the two polynomials, for all $n \geq 1$ we have 
\begin{equation*}\label{eq:setup:def-u-ratio}
u_{n} = r(n)u_{n-1},
\end{equation*}
and consequently the $n$th term $u_n$ is given by the following product:
\begin{equation}\label{eq:setup:def-u-prod}
u_{n} = u_0 \prod_{m=1}^{n} r(m) = u_0 \prod_{m=1}^{n} \frac{g(m)}{f(m)}.
\end{equation}

Dividing $f$ and $g$ by any common factors, we may freely assume that $f$ and $g$ are coprime. We will say that the recurrence \eqref{eq:setup:def-u} is \emph{regular} if additionally all the roots of $f$ and $g$ are distinct up to integer shifts, meaning that for each $\xi \in \CC$ with $f(\xi)g(\xi) = 0$ we have $f(\xi+d)g(\xi+d) \neq 0$ for all $d \in \ZZ \setminus \{0\}$. 
Of course, not all hypergeometric sequences are regular.
However, we can ensure regularity at the cost of introducing a rational factor.

\begin{lemma}\label{lem:setup:regularise}
	Let $\seq[\infty]{u_n}{n=0}$ be a hypergeometric sequence given by \eqref{eq:setup:def-u}. Then there exists a regular hypergeometric sequence $\seq[\infty]{\tilde u_n}{n=0}$ and a rational function $q(x) \in \QQ(x)$ such that $u_n = q(n)\tilde u_n$ for all $n \geq 0$. 
\end{lemma}
\begin{proof}
	Let $f(x) = \lambda \prod_{i=1}^I f_i(x)$ and $g(x) = \mu \prod_{j=1}^J g_j(x)$ be the factorisations of $f$ and $g$ into irreducible monic factors. Pick monic polynomials $h_k(x) \in \QQ[x]$ for $1 \leq k \leq K$ such that
    
\begin{enumerate}
\item for each polynomial $h \in \set{f_i}{1 \leq i \leq I} \cup \set{g_j}{1 \leq j \leq J}$ there exists $1 \leq k \leq K$ and $d \in \ZZ$ such that $h(x) = h_k(x+d)$;
\item the polynomials $h_k$, $1 \leq k \leq K$, are pairwise distinct up to integer shifts, meaning that there are no $1 \leq k < l \leq K$ and $d \in \ZZ$ such that $h_k(x) = h_l(x+d)$.
\end{enumerate}
In other words, we obtain $h_k$ ($1 \leq k \leq K$) by picking a {single} representative from each equivalence class of $f_i$ ($1 \leq i \leq I$) and $g_j$ ($1 \leq j \leq J$) with respect to an equivalence relation $\sim$ on $\QQ[x]$, where polynomials \(h,h'\in\QQ[x]\) are \emph{\(\sim\)-equivalent} if and only if $h$ and $h'$ differ by an integer shift, $h(x) = h'(x+d)$ for some $d \in \ZZ$.
Following our earlier observations, we may freely assume that each $h_k$ ($1 \leq k \leq K$) has no positive integer roots.
Pick $1 \leq k \leq K$ and let $I_k$ (resp.\ $J_k$) denote the set of those $1 \leq i \leq I$ (resp.\ $1 \leq j \leq J$) for which we have $h_k \sim f_i$ (resp.\ $h_k \sim g_j$). Let $\gamma_k = \# J_k - \# I_k$, $\tilde r(x) = \prod_{k=1}^{K} h_k(x)^{\gamma_k}$ and let $\tilde f(x),\tilde g(x) \in \QQ[x]$ be the coprime polynomials such that $\tilde r(x) = \tilde g(x)/\tilde f(x)$. More explicitly, $\tilde f$ and $\tilde g$ are given by
\begin{equation*}
	\tilde g(x) = \prod_{k=1}^K h_k^{\max(\gamma_k,0)}{(x)} \quad \text{and} \quad 
	\tilde f(x) = \prod_{k=1}^K h_k^{\max(-\gamma_k,0)}{(x)}.	
\end{equation*}

Let $\seq[\infty]{\tilde u_n}{n=0}$ be the hypergeometric sequence with $\tilde u_0 = u_0$ that satisfies the recurrence relation 
\begin{equation*}\label{eq:setup:def-u-tilde}
	\tilde f(n) \tilde u_{n} = \tilde g(n) \tilde u_{n-1}
\end{equation*}
for all $n \geq 1$.
By construction, $\seq[\infty]{\tilde u_n}{n=0}$ is regular. It remains to show that the ratio $\tilde u_n/u_n$ is a rational function of $n$. We have
\begin{equation*}
	\frac{\tilde u_n}{u_n} = \prod_{m=1}^n \prod_{k=1}^K 
	\bra{\prod_{j \in J_k} \frac{g_j(m)}{{h_k(m)}} } 
	\bra{\prod_{i \in I_k} \frac{f_i(m)}{{h_k(m)}} }^{-1} .
\end{equation*}
Working with each factor separately, it will suffice to show that for each $h \sim h'$ the product $\prod_{m=1}^n {h'(m)}/{h(m)}$ is a rational function of $n$, which is a simple consequence of the fact that all but a bounded number of terms in the product cancel out. Indeed, if $h(x) = h'(x+d)$ with $d \geq 0$ then   for all $n \geq d$
\begin{equation*}
\prod_{m=1}^n \frac{h'(m)}{h(m)} = \frac{\prod_{m=1}^n h'(m)}{\prod_{m=1}^n h'(m+d)} = \frac{\prod_{m=1}^d h'(m)}{\prod_{m=n+1}^{n+d} h'(m)},  
\end{equation*}
and one can check (either by backwards induction, or by direct computation) that the same formula holds for all $n \geq 0$. The case where $h(x) = h'(x+d)$ with $d < 0$ is entirely analogous.
An extended account is given in \cite[Appendix B]{nosan2022membership}.
\end{proof}

\subsection*{\texorpdfstring{\(p\)-adic}{p-adic} analysis}
In this subsection we briefly introduce the common notations and terminology that will be employed throughout the sequel.

For a non-zero rational number $r$, the \emph{square-free part of $r$} %
is the unique integer $d$ such that $r = q^2 d$ for some rational $q$. 
For the sake of completeness, we define the square-free part of $0$ to be $1$. We call an integer $n$ if it is not divisible by a square of any prime, meaning that it is equal to its square-free part.
As a quick example, consider $r=\tfrac{1}{8}$; then $q=\tfrac{1}{4}$ and $d=2$.

Let \(p\in\N\) be a prime. Denote by \(\nu_p\colon \QQ\to\ZZ\cup\{\infty\}\) the \emph{\(p\)-adic valuation} on \(\QQ\).
We recall that for every non-zero \(x\in\QQ\), the valuation \(\nu_p(x)\) is the unique integer for which the equality
    \(x = p^{\nu_p(x)}\tfrac{a}{b}\)
holds (where \(a,b\in\Z\) and \(p \nmid a,b\)). We define \(\nu_p(0) \coloneqq \infty\).
For a a rational number \(r\)  whose denominator is not divisible by $p$, we let $\rep(r)$ denote the representative of $r$ modulo $p$, that is, the unique integer in $\{0,1,\dots,p-1\}$ such that $\nu_p(r - \rep(r)) > 0$. 

For a prime $p$ and an integer \(n\), we let $\bra{\frac{n}{p}}$ denote the Legendre symbol.
Given a square-free integer $\Delta$ with $\bra{\frac{\Delta}{p}} = 1$, pick an integer $D$ with $D^2 \equiv \Delta \pmod p$ (for concreteness, we may require e.g.\ that $0 \leq D < p/2$). For a number $r + s \sqrt{\Delta} \in \QQ(\sqrt{\Delta})$ such that the denominators of $r$ and $s$ are not divisible by $p$, we let $\rep(r+s\sqrt{\Delta})$ denote the integer $\rep(r+sD)$. 

\subsection*{Galois theory}

Let \(K\) be a number field that is Galois over \(\QQ\). We let \(\mathcal{O}_K\) denote the ring of integers in \(K\).
Suppose that \(\mathfrak{p}\) is a prime of \(K\) lying over the rational prime \(p\in\Z\).
We call the subgroup \(D(\mathfrak{p}) \coloneqq \{\sigma \in \mathrm{Gal}(K/\QQ) \mid  \sigma(\mathfrak{p}) = \mathfrak{p}\}\) the \emph{decomposition group of \(\mathfrak{p}\)}.
It is known that each element of \(D(\mathfrak{p})\) acts in a well-defined way on the finite field \(\FF_\mathfrak{p} = \mathcal{O}_K / \mathfrak{p}\) and, in addition, that this action
 fixes \(\FF_{p}\), the finite field with \(p\) elements.
 
Thus each element of \(D(\mathfrak{p})\) is an element of \(\mathrm{Gal}(\FF_\mathfrak{p}/\FF_p)\).
When \(p\) is unramified in \(K\), it is well-known that the groups \(D(\mathfrak{p})\) and \(\mathrm{Gal}(\FF_\mathfrak{p}/\FF_p)\) are isomorphic.
Further, the group \(\mathrm{Gal}(\FF_\mathfrak{p}/\FF_p)\) is cyclic with a canonical choice of generator, the \emph{Frobenius element} \(\mathrm{Fr}_p \colon x \mapsto x^p\).
Lifting this element via the aforementioned isomorphism to \(D(\mathfrak{p})\) gives an element \(\mathrm{Fr}_\mathfrak{p}\).

\subsection*{Roots of quadratic congruences to prime moduli}
The investigation of the distribution of roots of quadratic polynomials modulo primes was initiated by Duke, Friedlander and Iwaniec \cite{dukefriedlanderiwaniec1995}, and continued by 
T\'{o}th \cite{Toth-2000} (see also Homma~\cite{homma2008} and Ngo~\cite{Ngo-2024+}).
We will say that a sequence $\seq{x_n}{n \in I}$ indexed by an infinite subset of $\NN$ is \emph{uniformly distributed} in $[0,1]$ if
\begin{equation*}
	\frac{ 
		\#\! \set{n \in I}{ n < N,\ x_n \in [\a,\b)}
	}{
		\#\! \set{n \in I}{ n < N}
	} \to \b-\a 
	 \text{ as } N \to \infty.
\end{equation*} 
Similarly, we say that a $k$-tuple of sequences $\langle x_n^{(i)}\rangle_{n \in I}$ ($1 \leq i \leq k$) is \emph{uniformly distributed} in $[0,1]$ if
\begin{equation*}
	\frac{ 
		\sum_{i=1}^k \# \{n \in I \mid n < N,\ x_n^{(i)} \in [\a,\b)\}
	}{
		k \cdot \#\! \set{n \in I}{ n < N}
	} \to \b-\a 
	 \text{ as } N \to \infty.
\end{equation*}

\begin{theorem}[T\'oth~\cite{Toth-2000}]\label{thm:toth}
	Let $\Delta \in \NN$ be square-free, $q \in \NN$, $a \in \ZZ/q\ZZ$ and let $P$ be the set of primes $p \equiv a \pmod q$ such that $\bra{\frac{\Delta}{p}} = 1$. 
    Let $P'\subset P$ be the subset of primes such that for $r,s \in \QQ$, the sequence $\seq{\rep(r \pm s \sqrt{\Delta})/p}{p \in P'}$ is well-defined.
    If \(P'\) is infinite, then for each \(r,s\in\QQ\) with \(s\neq 0\), $\seq{\rep(r \pm s \sqrt{\Delta})/p}{p \in P'}$ is {uniformly distributed} %
    in $[0,1]$.
\end{theorem}

In particular, in the situation above, for fixed $0 \leq \a < \b \leq 1$ and $\delta > 0$, for all sufficiently large $N$ there are $\gg N/\log N$ primes $p \in P$ with $N \leq p < (1+\delta)N$ such that $\rep(r \pm s \sqrt{\Delta})/p \in [\a,\b)$. 

\begin{corollary}\label{thm:toth-quant}
	Let $\Delta \in \NN$ be square-free, $q \in \NN$, $a \in \ZZ/q\ZZ$ and let $P$ be the set of primes $p \equiv a \pmod q$ such that $\bra{\frac{\Delta}{p}} = 1$. Let also $r,s \in \QQ$ with $s \neq 0$, $0 \leq \a < \b \leq  1$ and $\delta > 0$. If $P$ is infinite then there exist $N_0,c > 0$ such that for each $N \geq N_0$ there are at least $c N/\log N$ primes $p \in P$ such that $N \leq p < (1+\delta)N$ and $\rep(r + s \sqrt{\Delta})/p \in [\a,\b)$ or $\rep(r - s \sqrt{\Delta})/p \in [\a,\b)$. 
\end{corollary}
Let \(\Q^\times\) denote the multiplicative group of non-zero rational numbers.
For \(t = a/b\in\Q^\times\) with $a,b \in \ZZ$ coprime, $h_{\weil}(t)$ satisfies $h_{\weil}(t) = \max(\log |a|, \log |b|)$ cf.~\cite[pg.~167]{lang2002algebra}. 
\begin{proposition} \label{prop:heights}
The Weil height possesses the following properties properties.
\begin{enumerate}
    \item For \(\alpha_1,\alpha_2,\ldots, \alpha_k\in\QQ\), \(h_{\weil}(\alpha_1 \alpha_2 \cdots \alpha_k) \le h_{\weil}(\alpha_1) + h_{\weil}(\alpha_2) + \cdots + h_{\weil}(\alpha_k)\),
    \item \(h_{\weil}(\alpha^m) = |m|h_{\weil}(\alpha)\) for \(\alpha\in\QQ^\times\) and \(m\in\Z\), and
    \item     Let \(q\in\Q(x)\) be a rational function.  Then for \(\alpha\in\Q\), we have
        \(
            h_{\weil}(q(\alpha)) = \deg(q)h_{\weil}(\alpha) + O(1)
        \)
        where the implied constant depends only on \(q\). (Said constant can be explicitly estimated in terms of the heights of the coefficients defining \(q\).) Here the degree of the rational function \(q(x) = \hat{g}(x)/\hat{f}(x)\) where \(\hat{f}\) and \(\hat{g}\) are coprime polynomials is given by \(\deg(q) = \max\{\deg(f),\deg(g)\}\).
    \item We have \(2h_{\weil}(\alpha) \ge \sum_p |\nu_p(\alpha)| \log p\).
\end{enumerate}
\end{proposition}
\begin{proof}
The first three properties are standard in the literature cf.~{\cite[Proposition 3.2]{zannier2014diophantine} and \cite[pg.~7]{zannier2018heights}}.
 The final property follows straightforwardly from the observation that \(\log|x| = \sum_p \nu_p(x)\log p\). 
 Suppose that \(\alpha = a/b\) where \(a,b\in\Z\) are coprime, then we have
    \begin{equation*}
        2 h_{\weil}(\alpha) = 2\max(\log|a|,\log|b|) \ge \log|a| + \log|b| = \sum_p |\nu_p(\alpha)| \log p,
    \end{equation*}
as desired.
\end{proof}

In the sequel, we call a $p$-adic number $\a = \sum_{k=0}^\infty \a^{(k)}p^k$ \emph{normal} if the sequence of \(p\)-adic digits $\seq[\infty]{\a^{(k)}}{k=0}$ is normal, i.e., if for each $\ell \geq 1$ and each $\ell$-length pattern $w \in \{0,1,\dots,p-1\}^\ell$, the set of positions $\{s \in \NN_0 \mid \a^{(s+k)} = w_k \text{ for all } 0 \leq k < \ell\}$, has density $p^{-\ell}$.

\section{Growth estimates for the Weil height of hypergeometric sequences with quadratic factors} \label{sec:growth}
\newcommand{\discr}{\operatorname{Disc}}

\begin{proposition}
\label{prop:quadratic}
	Let $\seq[\infty]{u_n}{n=0}$ be a regular hypergeometric sequence given by \eqref{eq:setup:def-u}. Suppose that each irreducible factor of $fg$ has degree at most two, and let $\cD \subset \NN$ denote the discriminants of the quadratic factors of $f g$. Suppose further that there is $\Delta \in \cD$ and a prime $p$ such that {the following conditions hold,}
\begin{enumerate}
\item \label[condition]{cond:quadraticA} $\bra{\frac{\Delta}{p}} = 1$, and
\item \label[condition]{cond:quadraticB} $\bra{\frac{\Delta'}{p}} = -1$ for all $\Delta' \in \cD \setminus \{\Delta\}$.
\end{enumerate}	
{Then we have $h_{\weil}(u_n) \gg n$.}
\end{proposition}

\begin{remark}\label{rmk:re}
	In the situation of \cref{prop:quadratic}, let $p_1,p_2,\dots,p_r$ be the list of all primes dividing at least one of $\Delta \in \cD$. Note that for any $\epsilon_1,\epsilon_2,\dots,\epsilon_r \in \{0,1\}$ there exist infinitely many primes $p$ such that $\bra{\frac{p_i}{p}} = (-1)^{{\epsilon_i}}$ for all $1 \leq i \leq {r}$. %
    Indeed, it follows from quadratic reciprocity that there is a residue $p_0$ modulo $M := 4p_1 p_2 \cdots p_{{r}}$, %
    coprime to $M$, such that for each prime $p \equiv p_0 \pmod M$ we have $\bra{\frac{p_i}{p}} = \epsilon_i$ for all $1 \leq i \leq {r}$, %
    and existence of infinitely many such primes $p$ follows from Dirichlet's theorem. 
    For $\Delta \in \cD$, let $\delta^{(\Delta)} \in \{0,1\}^r$ be the vector given by $\delta^{(\Delta)}_j = 1$ if $p_j \mid \Delta$ and $\delta^{(\Delta)}_j = 0$ otherwise. 
    Bearing in mind the above discussion,	we see that for each $\Delta \in \cD$ the following are equivalent:
\begin{enumerate}
\item there exists at least one prime $p$ which satisfies \cref{cond:quadraticA,cond:quadraticB} in \cref{prop:quadratic};
\item there exists infinitely many primes $p$ which satisfies \cref{cond:quadraticA,cond:quadraticB} in \cref{prop:quadratic};
\item there exists $\epsilon \in \{0,1\}^r$ such that $\delta^{(\Delta)} \cdot \epsilon \equiv 0 \pmod 2$ and $\delta^{(\Delta')} \cdot \epsilon \equiv 1 \pmod 2$ for all $\Delta' \in \cD \setminus \{\Delta\}$.
\end{enumerate}
In particular, given a regular hypergeometric sequence $\seq[\infty]{u_n}{n=0}$ as in the above setting, verifying whether the conditions in \cref{prop:quadratic} {are satisfied by \(\seq[\infty]{u_n}{n=0}\)} is reduced to simple linear algebra.
From the observation in \cref{ex:Deltaconditions} (below), any regular hypergeometric sequence in class \(\mathscr{C}\) satisfies the conditions in \cref{prop:quadratic}.
\end{remark}

\begin{example} 
\label{ex:Deltaconditions}
Let \(\Delta_1\) and \(\Delta_2\) be integers such that $\Delta_1,\Delta_2$ and $\Delta_1 \Delta_2$ are not squares. Then \cref{cond:quadraticA,cond:quadraticB} in \cref{prop:quadratic} are satisfied if $\Delta_1 \in \cD \subset \{\Delta_1,\Delta_2,\Delta_1\Delta_2\}$. Moreover, \cref{cond:quadraticA,cond:quadraticB} are automatically satisfied if $\Delta_1,\Delta_2$ and $\Delta_1 \Delta_2$ are square-free.

More generally, let $\Delta_1,\Delta_2,\dots,\Delta_m$ be integers such that no non-empty product $\prod_{i \in I} \Delta_i$ with $\emptyset \neq I \subseteq \{1,2,\dots,m\}$ is a square. Then the condition in \cref{prop:quadratic} is satisfied if $\Delta_1 \in \cD$ and
\[ 
\cD \subset \set{\prod_{i \in I} \Delta_i}{ \emptyset \neq I \subset \{1,2,\dots,m\},\ {\#(I \setminus \{1\})} \equiv 1 \pmod 2}.
\]
To see this, it is enough to pick a prime $p$ such that $\bra{\frac{\Delta_1}{p}} = 1$ and $\bra{\frac{\Delta_i}{p}} = -1$ for $2 \leq i \leq m$. 
\end{example} 

It is a standard observation
that representations of non-integer rational numbers modulo large primes are never too small. Indeed, if $A/B$ ($A \in \ZZ,\ B \in \NN$) is a representation of a non-integer rational number in reduced form (meaning that $\gcd(A,B) = 1$ and $B \geq 2$) then for sufficiently large primes $p$, {there exists a positive constant \(\varepsilon(A,B)\)} such that
\begin{equation*}
	\rep\bra{\frac{A}{B}} = \frac{A+ip}{B} \geq \frac{p}{B}- {\varepsilon(A,B)} %
\end{equation*}
where $1 \leq i < B$ is specified by $A+ip \equiv 0 \pmod B$. %
We will need an analogous statement concerning elements of quadratic extensions of $\QQ$.

\begin{lemma}\label{lem:approx-eq}
Let $\Delta \in \NN$ be a square-free integer and $p$ a prime such that $\bra{\frac{\Delta}{p}} = 1$.
For some \(C>0\), let 
    $r_1,r_2,s_1,s_2 \in \QQ$ be rational numbers with Weil height at most $C$.
If \(\delta>0\),
\begin{equation}\label{eq:680:99}
	\rep(r_1 + s_1 \sqrt{\Delta}) < \delta p,	\quad \text{and} \quad \rep(r_2 + s_2 \sqrt{\Delta}) < \delta p,
\end{equation}
then there exist $r_1',r_2',r_0,s_0 \in \QQ$ with $r_1' - r_1, r_2' - r_2 \in \ZZ$ and $A_1,A_2 \in \NN$ with $\gcd(A_1,A_2) = 1$ such that
\begin{align*}
r_1' + s_1 \sqrt{\Delta} &= A_1(r_0+s_0\sqrt{\Delta}), &&&
r_2' + s_2 \sqrt{\Delta} &= A_2(r_0+s_0\sqrt{\Delta}), \\
\rep(r_1' + s_1 \sqrt{\Delta}) &= A_1 \rep(r_0+s_0\sqrt{\Delta}), &&&
\rep(r_2' + s_2 \sqrt{\Delta}) &= A_2 \rep (r_0+s_0\sqrt{\Delta}).
\end{align*}
\end{lemma}
\begin{proof}
Let $A_1,A_2 \in \ZZ$ be integers such that $\gcd(A_1,A_2) = 1$ and $A_2 s_1 = A_1 s_2$. For the ease of exposition, let us suppose that $A_1,A_2 \neq 0$ (the argument simplifies otherwise) and put $s_0  = s_1/A_1 = s_2/A_2$. Thus, we have
\begin{equation}\label{eq:680:00}
	 \rep\bra{A_1(r_1/A_1 + s_0 \sqrt{\Delta})} < \delta p,
	\quad \text{and} \quad \rep\bra{A_2(r_2/A_2 + s_0 \sqrt{\Delta})} < \delta p.
\end{equation}
    Let $B_1,B_2 \in \ZZ$ such that $A_1 B_2 - B_1 A_2 = 1$. By taking a linear combination of the two conditions in \eqref{eq:680:00} with weights $B_1,B_2$ we conclude that
\begin{equation}\label{eq:680:10}
	 \rep\bra{ s_0 \sqrt{\Delta} + r_0 } \ll_C \delta p,
	 \quad \text{or} \quad 
	 \rep\bra{ -s_0 \sqrt{\Delta} -r_0 }  \ll_C \delta p,
\end{equation}
for some rational number $r_0 \in \QQ$ with Weil height $O(C)$. 
For concreteness, assume that the first condition in \eqref{eq:680:10} holds.
Substitution into \eqref{eq:680:00} leads to:
\begin{equation*} %
	 \rep\bra{A_1(r_1/A_1 - r_0)}  \ll_C \delta p,
	 \quad \text{and} \quad  \rep\bra{A_2(r_2/A_{{2}} - r_0)} \ll_C \delta p.
\end{equation*}
Picking sufficiently small $\delta$ we can ensure that $p$ is large as a function of $C$, and hence the preceding asymptotic bound implies that 
\begin{equation*} %
	 r_1/A_1 - r_0 = m_1/A_1 \quad \text{and} \quad r_{{2}}/A_{{2}} - r_0 = m_2/A_{{2}},
\end{equation*}
for some $m_1,m_2 \in \ZZ$. Put $r_1' = r_1 - m_1$ and $r_2' = r_2 - m_2$. We have 
	\( r_1'/A_1 = r_2'/A_2 = r_0 \),
meaning that 
\begin{equation} \label{eq:680:50}
	 r_1' + s_1 \sqrt{\Delta} = A_1 (r_0 + s_0 \sqrt{\Delta}),
	 \quad \text{and} \quad r_2' + s_2 \sqrt{\Delta} = A_2 (r_0 + s_0 \sqrt{\Delta}).
\end{equation}
Bearing in mind \eqref{eq:680:10} and assuming that $\delta$ is sufficiently small, we conclude from the first equation in \eqref{eq:680:50} that
\begin{equation*} %
	\rep\bra{r_1' + s_1 \sqrt{\Delta} } = 
	\begin{cases}
	A_1 \rep\bra{r_0 + s_0 \sqrt{\Delta}} & \text{if } A_1 > 0;\\
	p- \abs{A_1} \rep\bra{r_0 + s_0 \sqrt{\Delta}} & \text{if } A_1 < 0.	
	\end{cases}
\end{equation*}
Comparing this with \eqref{eq:680:99} we see that, necessarily, $A_1 > 0$. Applying the same reasoning to the second equation in \eqref{eq:680:50} completes the argument.
\end{proof}

\begin{proof}[Proof of \cref{prop:quadratic}]

Let us decompose
\begin{equation}\label{eq:427:00}
	r(x) = \frac{g(x)}{f(x)} = \prod_{i=1}^I \ell_i(x)^{k_i} \prod_{j=1}^J q_j(x)^{m_i} r_0(x),
\end{equation}
where the exponents $m_i,k_j$ are integers, $\ell_i$ are monic linear polynomials, $q_j$ are monic quadratic polynomials with discriminant $\Delta$, and $r_0$ is a rational function whose numerator and denominator are products of quadratic polynomials with discriminants different from $\Delta$. 
Hereafter we freely assume that \(\Delta\) is square-free.  This assumption comes at no cost because the Legendre symbol is completely multiplicative in the top argument \cite[Theorem 9.3]{apostol2010number}.
For each $1 \leq i \leq I$ let $t_i$ be the root of $\ell_i$, meaning that 
\begin{equation*}
	\ell_i(x) = x - t_i.
\end{equation*}
Similarly, for each $1 \leq j \leq J$ let $r_j \pm s_j \sqrt{\Delta}$ ($s_j > 0$) be the roots of $q_j$, meaning that 
\begin{equation*}
	q_j(x) = \bra{x-r_j-s_j \sqrt{\Delta}}\bra{x-r_j + s_j \sqrt{\Delta}}.
\end{equation*}

Note that for each prime $p$ that satisfies \cref{cond:quadraticA,cond:quadraticB} in \cref{prop:quadratic} and which is larger than all primes appearing in denominators of $t_i$ ($1 \leq i \leq I$) and $r_j,s_j$ ($1 \leq j \leq J$), each of the linear factors $\ell_i$ has one root modulo $p$, each of the quadratic factors $q_j$ has two roots modulo $p$, and the numerator and denominator of $r_0$ have no roots modulo $p$.
Reordering the quadratic factors $q_j$ if necessary, we may further assume that $s_1 \leq s_j$ for all $1 \leq j \leq J$. 

Let $M$ be a positive integer that is sufficiently {multiplicatively rich} (i.e., divisible by a suitably constructed integer $M_0$) such that
\begin{enumerate}
\item there is a residue $p_0$, coprime to $M$, such that all primes $p$ with $p \equiv p_0 \pmod{M}$ satisfy \cref{cond:quadraticA,cond:quadraticB} in \cref{prop:quadratic} (cf.~\cref{rmk:re});
\item $Mt_i$ ($1 \leq i \leq I$) and $M r_j$, $M s_j$ ($1 \leq j \leq J$) are integers.
\end{enumerate} 
Let $\delta,\e > 0$ be small constants, to be specified in the course of the argument. We will be interested in primes $p$ satisfying the following conditions:
\begin{enumerate}[label = C.\roman*.]
\item \label[condition]{it:61:B} $p \equiv p_0 \pmod M$;
\item \label[condition]{it:61:A} $n < \delta p < (1+\e/3)n$;
\item \label[condition]{it:61:C} $\rep(r_1 + s_1 \sqrt{\Delta}) \in ((1-\e)n, n)$.
\end{enumerate}
Note that \cref{it:61:A} is equivalent to
\(
	p \in \big(\tfrac{1}{\delta}n,\tfrac{1+\e}{\delta}n\big)
\)
and that \cref{it:61:C} is implied by 
\[
	\rep(r_1 + s_1 \sqrt{\Delta}) \in \big( (1-\e)\delta p,  \big(1-\tfrac{2}{3}\e\big) \delta p\big).
\]
Hence, it follows from \cref{thm:toth-quant} that the number of primes satisfying \cref{it:61:A,it:61:B,it:61:C}, for fixed $\delta,\e > 0$, is $\gg n/\log n$ (here the implied constant depends on $\delta$ and $\e$).
We plan to show that for each prime $p$ satisfying \cref{it:61:B,it:61:A,it:61:C} we have $\nu_p(u_n) \neq 0$. Once this is accomplished, we obtain
\begin{equation*}
	h_{\weil}(u_n) \ge \frac{1}{2} {\sum_{p} |{\nu_p(u_n)}| \log p} \gg \frac{n}{\log n} \log n = n.
\end{equation*}
Here the first inequality follows directly from \cref{prop:heights}. 
We can absorb the factor of \(1/2\) into the implied asymptotic constant in order to finish the argument. 

We will consider contributions to $\nu_p(u_n)$ coming from different terms in \eqref{eq:427:00} separately. Since $\nu_p(r_0(m)) = 0$ for all $m \in \ZZ$, there is no contribution from $r_0$:
	\(\nu_p\bra{ \prod_{m=1}^n r_0(m) } = 0\).
 For each $1 \leq i \leq I$ we can explicitly describe the root of $\ell_i$ modulo $p$:
\begin{equation*}
	\rep(t_i) = \frac{T_i}{M}(p-p_0),
\end{equation*}
where $T_i$ is an integer independent of $p$. Thus, as long as $\delta < 1/M$, for sufficiently large $n$ we have $\rep(t_i) > n$ and consequently
	\(\nu_p\bra{ \prod_{m=1}^n \ell_i(m) } = 0\).
\cref{it:61:C} guarantees that $q_1$ has (at least) one root modulo $p$ that is less than $n$ and consequently
	\(\nu_p\bra{ \prod_{m=1}^n q_1(m) } \geq 1\).
(In fact, it is not hard to rule out the possibility that the other root of $q_1$ is also less than $n$, but we will not need this.)

It remains to show that for each $2 \leq j \leq J$, the contribution from $q_j$ to $\nu_p(u_n)$ is zero. For the sake of contradiction, suppose that for some $2 \leq j \leq J$ we have
	\(\rep\bra{r + s \sqrt{\Delta}} \leq n\),
where $r = r_j$ and $s = s_j$ or $s = -s_j$. This in turn implies that 
\begin{equation}\label{eq:690:21}
	\rep\bra{r + s \sqrt{\Delta}} \leq \delta p.
\end{equation}
If $\delta$ was chosen sufficiently small (as a function of $r_1,s_1,r_j,s_j$), then we conclude from \cref{lem:approx-eq} that $r+s\sqrt{\Delta}$ and $r_1 +s_1 \sqrt{\Delta}$ are both multiples of the same element of $\QQ(\sqrt{\Delta})$ with a small representative modulo $p$ (perhaps after an integer shift). In other words, we can find coprime positive integers $A,A_1$ and rational numbers $r',r_1',r^*,s^*$ such that $r' - r, r_1' - r_1$ are integers, and
\begin{align*}
	r' + s \sqrt{\Delta} &= A (r^* + s^*\sqrt{\Delta}),&
	r'_1 + s_1 \sqrt{\Delta} &= A_1 (r^* + s^*\sqrt{\Delta}),	\\
	\rep(r' + s \sqrt{\Delta}) &= A \rep(r^* + s^*\sqrt{\Delta}),&
	\rep(r'_1 + s_1 \sqrt{\Delta}) &= A_1 \rep(r^* + s^*\sqrt{\Delta}).	
\end{align*}
Since $s_1 \leq s_j$, we have $A_1 \leq A$. On the other hand, \cref{it:61:C} combined with \eqref{eq:690:21} implies that $A \leq (1+\e)A_1$. Assuming that $\e$ was chosen sufficiently small, this is only possible if $A = A_1$ and hence $s = s_1 = s_j$. It follows that 
\begin{align*}
	r' + s \sqrt{\Delta} &= A (r^* + s^*\sqrt{\Delta}) = 
	r'_1 + s_1 \sqrt{\Delta}.
\end{align*}
However, this implies that $r' - r_1$ is an integer, contradicting the assumption that $\seq[\infty]{u_n}{n=0}$ was regular.
\end{proof}

Our main result in this section, \cref{thm:main} (restated below), follows directly from \cref{prop:quadratic,prop:heights}.
\theoremmain*
\begin{proof}
Suppose that \(\seq[\infty]{u_n}{n=0}\) belongs to class \(\mathscr{C}\) and write \(u_n = q_n \tilde{u}_n\) where \(q\in\Q(x)\) is a rational function and \(\seq[\infty]{\tilde{u}_n}{n=0}\) is a regular hypergeometric sequence in class \(\mathscr{C}\). 
By the first two properties in \cref{prop:heights}, we have 
    \begin{equation*}
        h_{\weil}(u_n) \ge h_{\weil}(\tilde{u}_n) - h_{\weil}(1/q_n) = h_{\weil}(\tilde{u}_n) - h_{\weil}(q_n).
    \end{equation*}
The desired result then straightforwardly follows from two observations.
First, we can apply the growth estimate \(h_{\weil}(\tilde{u}_n) \gg n\) in \cref{prop:quadratic} since \(\seq[\infty]{\tilde{u}_n}{n=0}\) is regular and in class \(\mathscr{C}\).  Second, we can estimate \(h_{\weil}(q(n)) = \deg(q)\log n + O(1)\) from the third property in \cref{prop:heights}.
\end{proof}

\begin{remark}
    It is worth noting that the growth estimate in \cref{thm:main} still holds when we relax the assumption that \(\seq[\infty]{u_n}{n=0}\) is a member of \hyperlink{defin:C}{class~\(\mathscr{C}\)} to the assumptions present in \cref{prop:quadratic}.
    Our focus on the class of hypergeometric sequences with quadratic parameters, is motivated by the particular attention paid to this class in the literature~ \cite{hongarxiv2016, kenison2023membership, kenison2024threshold}.
    We shall return to this class in \cref{sec:membership} to discuss the Membership Problem in this setting.
\end{remark}

\section{Divergence of \texorpdfstring{$p$}{p}-adic valuations}\label{sec:divergence}

Let \(\seq[\infty]{u_n}{n=0}\) be a hypergeometric sequence satisfying \eqref{eq:setup:def-u} and $p$ a Hensel prime for $fg$. 
Following~\cite[Section~4]{kenison2023membership}, we say that $\seq[\infty]{u_n}{n=0}$ is \emph{$p$-symmetric} if $f$ and $g$ have the same number of roots in $\mathbb{F}_p$, and %
\emph{$p$-asymmetric} otherwise. 
In~\cite[Lemma~9]{kenison2023membership} it is shown that if $\seq[\infty]{u_n}{n=0}$ is $p$-asymmetric then $\nu_p(u_n) \to \pm \infty$ as $n \to \infty$, and the rate of divergence can be quantified.

In this section we shall first recall the statement of Lemma~9 in~\cite{kenison2023membership};
we then prove a converse of said result subject to a normality assumption on the roots of \(fg\) (\cref{lem:normal}).
For restricted classes of hypergeometric sequences, we also prove that \(p\)-symmetry holds if and only if certain properties of the Galois group associated with \(fg\) hold (\cref{prop:large:Galois,prop:quadratic,prop:cyclic}).

Let us recall the following result on the divergence of \(p\)-adic valuations for \(p\)-asymmetric hypergeometric sequences.
\begin{lemma}[{\cite[Lemma~9]{kenison2023membership}}]
\label{lem:pasymmetric-divergence}
Let \(\seq[\infty]{u_n}{n=0}\) be a hypergeometric sequence that satisfies \eqref{eq:setup:def-u}
and suppose that \(\seq[\infty]{u_n}{n=0}\) is \(p\)-asymmetric for some prime \(p\).
Let \(m_f\) denote the number of roots of \(f\) modulo \(p\) and define \(m_g\) similarly.
Then
    \begin{equation*}
        |\nu_p(u_n)| = \frac{|m_g - m_f|n}{p-1} + O(\log n)
    \end{equation*}
    where the implied constant depends only on \(fg\) and \(p\). In particular, $\nu_p(u_n) \to \infty$ as $n \to \infty$.
\end{lemma}

Recall that Borel's conjecture predicts that every irrational algebraic number is normal~\cite{borel1909, borel1950}.
In the sequel, we shall refer to the following \(p\)-adic prediction. 
This prediction is a weak version of the \(p\)-adic version of Borel's conjecture in~\cite[Conjecture 1.1]{sun2010borel}.
    \begin{conjecture} \label{conj:borel}
        Let \(\alpha\in\Z_p\) be an irrational algebraic number.  Then \(\alpha\) is normal in base \(p\).
    \end{conjecture}
A close inspection of the argument in \cite{kenison2023membership} shows that $p$-asymmetry is equivalent to divergence of $p$-adic valuations of $\seq[\infty]{u_n}{n=0}$ subject to the prediction concerning \(p\)-adic expansions in \cref{conj:borel}.
To be more precise, the following result (in combination with the previous discussion) shows that if \(\seq[\infty]{u_n}{n=0}\) is a hypergeometric sequence given by \eqref{eq:setup:def-u} and $p$ is a sufficiently large prime then, so long as all the $p$-adic roots of $fg$ are normal in base \(p\), $p$-asymmetry of \(\seq[\infty]{u_n}{n=0}\) is equivalent to divergence of the $p$-adic valuation of $u_n$.

\begin{lemma}\label{lem:normal}
	Let $\seq[\infty]{u_n}{n=0}$ be a hypergeometric sequence given by \eqref{eq:setup:def-u} and let $p$ be a Hensel prime for $fg$. Suppose that $|{\nu_p(u_n)}| \to  \infty$ as $n \to \infty$ and that all the roots of $fg$ in $\ZZ_p$ are normal. 
    Then $\seq[\infty]{u_n}{n=0}$ is $p$-asymmetric.
\end{lemma}
\begin{proof}   
We shall prove the contrapositive statement. 
Suppose that \(\seq[\infty]{u_n}{n=0}\) is $p$-symmetric. We want to show that $|{\nu_p(u_n) }|$ does not diverge to $\infty$ as $n \to \infty$. 
For our purposes, we will evaluate $\nu_p(u_n)$ for $n = p^s$ for suitably chosen $s \geq 0$.
From the product formula \eqref{eq:setup:def-u-prod}, it follows that
\begin{equation*}
    \nu_p (u_{p^s}) = \nu_p(u_0) + \nu_p \left( \prod_{n=1}^{p^s} f(n) \right) - \nu_p \left( \prod_{n=1}^{p^s} g(n) \right).
\end{equation*}
Let us focus on the contribution to the \(p\)-adic valuation made by the polynomial coefficient \(f\).
Since \(p\) is a Hensel prime of \(f\) such that \(f\) has \(m\) roots modulo \(p\), by Hensel's Lemma~\cite[Theorem 3.4.1]{gouvea2020padic}, there is a factorisation of \(f\) of the form
\(
        f(x) = (x-\alpha_1)\cdots (x-\alpha_m) h_f(x) \pmod{p^s}
\)
for each \(s>1\) where \(h_f\) has no zero modulo \(p\).
We have
\begin{equation*}
    \nu_p \left( \prod_{n=1}^{p^s} f(n) \right) = \sum_{n=1}^{p^s} \nu(f(n))
        = \sum_{n=1}^{p^s} \sum_{i=1}^m \nu(n-\alpha_i)
\end{equation*}
We let \(\a_i = \sum_{k=0}^\infty \a_i^{(k)}p^k\), \(1 \leq i \leq m\), denote the \(p\)-adic expansion of each of the roots of \(f\) in \(\ZZ_p\).
For each $1 \leq i \leq m$, let $\delta_i^{(s)}$ denote the largest index such that $\a_i^{(r+k)} = 0$ for $0 \leq k < \delta_i^{(s)}$ 

Let \(\tau_r\)
denote the \(r\)th level truncation map for \(p\)-adic expansions, so that \(\tau_r(\alpha_i) = \sum_{k=0}^{r-1} \alpha_i^{(k)}p^k\).
For \(1\le r \le s\), \(\nu(n-\alpha_i)\ge r\) if \(n = \tau_r(\alpha_i) + \ell p^t\) with \(\ell \in \{0,1,\ldots, p-1\}\) and \(r\le t \le s\). 
Thus the set of such \(n\) decomposes as a finite union of arithmetic progressions with common differences \(p^r, p^{r+1}, \ldots, p^s\).
We denote the characteristic function for each such arithmetic progression by \(\chi_{\{p^r \mid \wc\}}\) and note that each  progression contains \(p^{s-r}\) elements.
We also observe that the indices \(1\le n \le p^s\) for which \(v_p(n-\alpha_i)\ge s\) have \(p\)-adic digit expansions \(\sum_{k=s+\delta_i^{(s)}}^\infty \alpha_i^{(k)} p^k\) with \(1\le i \le m\).
It follows straightforwardly from the above observations that
\begin{equation*}
   \sum_{n=1}^{p^s} \sum_{i=1}^m \nu(n-\alpha_i) = 
      \sum_{n=1}^{p^s} \sum_{i=1}^m \sum_{r=1}^{s} \chi_{\{p^r \mid \wc\}} (n-\alpha_i) \; + \sum_{i=1}^m {(s+\delta_i^{(s)})} = \sum_{r=1}^s m p^{s-r} + \sum_{i=1}^m {(s+\delta_{i}^{(s)})}.
\end{equation*}

There is an analogous factorisation of \(g(x) = (x-\beta_1) \cdots (x-\beta_m) h_g(x) \pmod{p^s}\) and
we can apply the above reasoning to $g$ and its roots in \(\Z_p\) with \(p\)-adic expansions
\(
\b_i = \sum_{k=0}^\infty \b_i^{(k)}p^k\) for \(
1 \leq i \leq m\).
We also define $\gamma_i^{(s)}$ in an analogous manner to \(\delta_i^{(s)}\).

Then
\(	\nu_p(u_{p^s}) = \sum_{i=1}^m \delta_i^{(s)} - \sum_{i=1}^m \gamma_i^{(s)} \).
In particular, we have
\begin{equation*}
	|{\nu_p(u_{p^s})}| \leq \max\bra{ \sum_{i=1}^m \delta_i^{(s)}, \sum_{i=1}^m \gamma_i^{(s)}}.
\end{equation*}
Since each \(\alpha_i\) is normal, for each \(i\), $ \delta_i^{(s)}$ is bounded on average. 
More precisely, since for $\ell \geq 0$ we have $\delta_i^{(s)} \geq \ell$ if and only if $\alpha_i^{(s)} = \alpha_i^{(s+1)} = \dots = \alpha_i^{(s+\ell-1)} =0$, normality of $\alpha_i$ implies that 
\begin{equation*}
	\mathrm{dens} \bra{ \set{ s \in \NN }{ \delta_i^{(s)} \geq \ell } } = p^{-\ell}.
\end{equation*}
Letting $\ell$ be the least integer with $p^{\ell} > 2m$, applying the union bound we conclude that there are infinitely many values of $s$ such that $|\nu_p(u_{p^s})| \leq m \ell \leq m \log_p( 2m + 1)$. In particular, we have
	\(\liminf_{n \to \infty} |\nu_p(u_n)| < \infty\),
as desired.
\end{proof}
 
In \cite{kenison2023membership} it is shown that if $f$ and $g$ have different splitting fields then $\seq[\infty]{u_n}{n=0}$ is $p$-asymmetric for infinitely many primes.
Thus, by \cref{lem:pasymmetric-divergence}, in our divergence analysis we may restrict our attention to the situation where the splitting fields of $f$ and $g$ are the same. 
Considering a prime $p$ such that $f$ and $g$ split completely modulo $p$, we see that if $\seq[\infty]{u_n}{n=0}$ is $p$-symmetric then $\deg f = \deg g$. 
The following result is a straightforward consequence of the Chebotarev density theorem.
\begin{proposition}\label{prop:large:Galois}
Let $\seq[\infty]{u_n}{n=0}$ be a hypergeometric sequence given by \eqref{eq:setup:def-u} and suppose that $f$ and $g$ have the same splitting field $K$. Let $\a_1,\a_2,\dots,\a_d \in K$ and $\b_1,\b_2,\dots,\b_d \in K$ be the roots of $f$ and $g$ respectively. Then the following conditions are equivalent:
\begin{enumerate}
\item $\seq[\infty]{u_n}{n=0}$ is $p$-symmetric for all sufficiently large primes $p$;
\item \label{it:large:root-count} for each $\sigma \in \mathrm{Gal}(K/\QQ)$ we have
\begin{equation}\label{eq:large:root-count}
	\#\! \set{ 1 \leq i \leq d }{ \sigma(\a_i) = \a_i } 
=	\#\! \set{ 1 \leq i \leq d }{ \sigma(\b_i) = \b_i }\!.
\end{equation}
\end{enumerate}
\end{proposition}
\begin{proof}
	Let $p$ be a sufficiently large prime and \(\FF_p\) the finite field of \(p\) elements; in particular, we assume that $p$ is unramified in $K$. 
    The number of roots of $f$ in $\FF_p$ is the number of roots $\a_i$ of $f$ ($1 \leq i \leq d$) fixed by the corresponding Frobenius element $\mathrm{Fr}_p {\colon x \mapsto x^p}$. Thus, $\seq[\infty]{u_n}{n=0}$ is $p$-symmetric if and only if 
\begin{equation*}
	\#\! \set{ 1 \leq i \leq d }{ \mathrm{Fr}_p(\a_i) = \a_i } 
=	\#\! \set{ 1 \leq i \leq d }{ \mathrm{Fr}_p(\b_i) = \b_i}\!.
\end{equation*}
Let \(\mathfrak{C}\) be a conjugacy class of \(\mathrm{Gal}(K/\QQ)\).
By Chebotarev's density theorem, the set of primes \(p\) that do not divide the discriminant of \(K\) and for which \(\mathrm{Fr}_p\) belongs to \(\mathfrak{C}\) has positive density equal to \(\#\mathfrak{C}/\#\mathrm{Gal}(K/\QQ)\).  Importantly, there are infinitely many primes \(p\) for which \(\mathrm{Fr}_p\in\mathfrak{C}\).
\end{proof}

As an immediate consequence of \cref{prop:large:Galois}, we see that in order for $\seq[\infty]{u_n}{n=0}$ to be $p$-symmetric for all sufficiently large primes $p$ it is enough that the respective roots $\a_1,\a_2,\dots,\a_d$ and $\b_1,\b_2,\dots,\b_d$ of $f$ and $g$ satisfy (possibly after rearrangement):
\begin{equation}\label{eq:large:same-roots}
	\QQ(\a_i) = \QQ(\b_i) \quad  \text{for all } 1 \leq i \leq d. 
\end{equation}
{Further, if \eqref{eq:large:same-roots} does not hold for each \(\sigma\in\Gal(K/\Q)\), then we straightforwardly deduce that the set of primes $p$ for which $\seq[\infty]{u_n}{n=0}$ is $p$-asymmetric has positive relative density in the set of all primes.}
We recall \hyperlink{defin:D}{class \(\mathscr{D}\)} (\cref{sec:intro}) of hypergeometric sequences \(\seq[\infty]{u_n}{n=0}\) where the condition \eqref{eq:large:same-roots} does not hold.  

The next example shows that the condition in \eqref{eq:large:same-roots} is not necessary for \(p\)-symmetry.%

\begin{example}\label{ex:large}
	Let $f,g \in \ZZ[X]$ be given by
	\begin{align*}
	f(X) &= (X^4-10X^2+1)X^2,\\ g(X) &= (X^2-2)(X^2-3)(X^2-6).
	\end{align*}		
	Then $f$ and $g$ have the same number of roots modulo each prime $p \geq 5$. Indeed, since $6 = 2 \cdot 3$, for each prime $p$ either all or exactly one of $2,3$ and $6$ are quadratic residues modulo $p$. The roots of $f$ are $\pm \sqrt{2} \pm \sqrt{3}$ so $f$ has $4$ roots modulo $p$ if $2$ and $3$ are quadratic residues modulo $p$, and no roots otherwise. Thus, if $2,3$ and $6$ are quadratic residues modulo $p$ then $f$ and $g$ both have $6$ roots modulo $p$, and if only one of $2,3$ and $6$ is a quadratic residue modulo $p$ then $f$ and $g$ both have two roots modulo $p$. However, the roots of $f$ and $g$ cannot be rearranged so that \eqref{eq:large:same-roots} holds. 
\end{example}

Despite \cref{ex:large}, there are situations where $p$-symmetry implies \eqref{eq:large:same-roots}.

\begin{proposition} \label{prop:intermediate}
	Let $\seq[\infty]{u_n}{n=0}$ be a hypergeometric sequence given by \eqref{eq:setup:def-u} and suppose that all irreducible factors of $f$ and $g$ have degree at most $2$. If $\seq[\infty]{u_n}{n=0}$ is $p$-symmetric for all sufficiently large primes $p$ then \eqref{eq:large:same-roots} holds (possibly after permuting the roots).
\end{proposition}
\begin{proof}
	Let $d = \deg f = \deg g$. It will be convenient to assume that $f$ and $g$ can be written as products $\prod_{i=1}^{d/2} f_i$ and $\prod_{i=1}^{d/2} g_i$ where $f_i$ and $g_i$ are (not necessarily irreducible) polynomials of degree $2$; if $d$ is even we can simply group the linear terms into pairs, and if $d$ is odd we can multiply $f$ and $g$ by the same monomial and reduce to the previous case. Let $\Delta_i$ ($1 \leq i \leq d/2$) be square-free integers such that the discriminant of $f_i$ takes the form $q^2 \Delta_i$ for some rational $q$ (if the discriminant of $f_i$ is $0$, put $\Delta_i = 1$). Note in particular that $\Delta_i = 1$ if $f_i$ splits over $\QQ$. Let $\Gamma_i$ be defined analogously, with $g_i$ in place of $f_i$. Let $p_1,p_2, \dots, p_r$ be the list of primes that divide $\Delta_i$ or $\Gamma_i$ for at least one $i$ and let $K = \QQ(\sqrt{p_1},\sqrt{p_2},\dots,\sqrt{p_r})$. Note that $f$ and $g$ split completely over $K$. For $\epsilon \in \{0,1\}^r$ let $\sigma_{\epsilon} \in \mathrm{Gal}(K/\QQ)$ be the automorphism of $K$ specified by $\sigma_{\epsilon}(\sqrt{p_j}) = (-1)^{\epsilon_j} \sqrt{p_j}$, and let $L_\epsilon < K$ be the field given by $L_{\epsilon} = \set{x \in K}{ \sigma_{\epsilon}(x) = x}$. By  \cref{prop:large:Galois}, for each $\epsilon \in \{0,1\}^r$ we have 
\begin{equation}\label{eq:large:58:99}
	\# \set{1 \leq i \leq d}{ \a_i \in L_{\epsilon}} = 
	\# \set{1 \leq i \leq d}{ \b_i \in L_{\epsilon}}.
\end{equation}

	For $1 \leq i \leq d$, let $\delta^{(i)} \in \{0,1\}^r$ be given by $\delta^{(i)}_j = 1$ if $p_j \mid \Delta_i$ and $\delta^{(i)}_j = 0$ otherwise, meaning that 
    \(\Delta_i = \prod_{j=1}^r p_j^{\delta^{(i)}_j}\).
    Let $\gamma^{(i)}_j$ be defined analogously, with $\Gamma_i$ in place of $\Delta_i$. 
    With this notation, we have $\alpha_i \in L_{\epsilon}$ if and only if $\delta^{(i)} \cdot \epsilon \equiv 0 \pmod 2$, where we use the convention $\delta^{(i)} \cdot \epsilon = \sum_{j=1}^r \delta^{(i)}_j \epsilon_j$. 
    Thus,
\begin{equation*} %
	\frac{1}{2}\bra{ 1 + (-1)^{\delta^{(i)} \cdot \epsilon}} = 
	\begin{cases}
		1 & \alpha_i \in L_{\epsilon},\\
		0 & \alpha_i \not \in L_{\epsilon}.
	\end{cases}
\end{equation*}
Pick any $\tau \in \{0,1\}^r \setminus \{0^r\}$. 
Multiplying the preceding equation by $(-1)^{\tau \cdot \epsilon}$, taking the average over $\epsilon \in \{0,1\}^r$, and summing over all $1 \leq i \leq d$ we see that
\begin{equation*} %
	\Eop_{\epsilon \in \{0,1\}^r} (-1)^{\tau \cdot \epsilon} \#\! \set{ 1 \leq i \leq {d}  }{ \a_i \in L_{\epsilon}} = \frac{1}{2} \# \{i \mid \delta^{(i)} = \tau \}.
\end{equation*}
For $\tau = 0^r$ we have a similar formula with an additional term $d/2$ on the right side. 
Applying the same reasoning to $\gamma^{(i)}$ rather than $\delta^{(i)}$ and bearing in mind \eqref{eq:large:58:99} we conclude that for each $\tau \in \{0,1\}^r$ we have
\begin{equation*} %
\# \{1 \leq i \leq d \mid \delta^{(i)} = \tau \} = \# \{1 \leq i \leq d \mid \gamma^{(i)} = \tau \}.
\end{equation*}
Thus, possibly after rearranging the roots of $f$ and $g$, we have $\delta^{(i)} = \gamma^{(i)}$ for all $1 \leq i \leq d$. It follows that we also have $\Delta_i = \Gamma_i$ and $\QQ(\alpha_i) = \QQ(\beta_i)$, as needed.
\end{proof}

For the following proposition, recall that a field $K$ is a \emph{cyclic extension} of $\QQ$ if $K$ is a Galois extension of $\QQ$ and the Galois group $\mathrm{Gal}(K/\QQ)$ is cyclic. 
For instance, if $p$ is a prime and $\zeta_p = \exp(2\pi \iu/p)$ is a $p$th root of unity then %
$\QQ(\zeta_p)$ is a cyclic extension of $\QQ$ since $\mathrm{Gal}(\QQ(\zeta_p)/\QQ) \simeq (\ZZ/p\ZZ)^{\times} \simeq \ZZ/(p-1)\ZZ$~\cite[Theorem 20.12]{stewart2015galois}.

\begin{proposition} \label{prop:cyclic}
	Let $\seq[\infty]{u_n}{n=0}$ be a hypergeometric sequence given by \eqref{eq:setup:def-u} and suppose that $f$ and $g$ split completely over a cyclic extension $K$ of $\QQ$. If $\seq[\infty]{u_n}{n=0}$ is $p$-symmetric for all sufficiently large primes $p$ then \eqref{eq:large:same-roots} holds (possibly after a permutation of the roots).
\end{proposition}
\begin{proof}
	Since all subgroups of the cyclic group $\mathrm{Gal}(K/\QQ)$ are cyclic (and hence generated by a single element), the fundamental theorem of Galois theory implies that each subfield $L$ of $K$ takes the form $L = \set{x \in K}{\sigma(x) = x}$ for some $\sigma \in \mathrm{Gal}(K/\QQ)$. Hence, it follows from \cref{prop:large:Galois} that $f$ and $g$ have the same number of roots in $L$. Note that for $\a \in L$ we have $\QQ(\a) = L$ if and only if 
	\[ \a \in L \setminus \textstyle \bigcup_{M < L} M,\]
	 where the union is taken over all proper subfields $M < L$. We may enumerate these subfields as $M_1, M_2, \dots, M_r$. By the inclusion-exclusion principle, we have
	 \begin{equation*}
	 \#\!\set{ 1 \leq i \leq d }{ \QQ(\a_i) = L } = \sum_{S \subset \{1,2,\dots,r\} } (-1)^{|S| }
	 \#\!\set{ 1 \leq i \leq d  }{\begin{array}{l} \a_i \in L \text { and }   \a_i \in M_k\\ \!\!\text{ for all } k \in S\end{array}}. 
	 \end{equation*}

Since an intersection of fields is also a field, applying the same reasoning to the roots of $g$ we conclude that 
\begin{equation*}
\#\!\set{ 1 \leq i \leq d  }{ \QQ(\a_i) = L } = \#\!\set{ 1 \leq i \leq d  }{ \QQ(\b_i) = L }\!. 
\end{equation*}
This implies that we can reorder $\a_i$ and $\b_i$ in such a way that \eqref{eq:large:same-roots} holds.
\end{proof}

Finally, we observe that \cref{thm:asymmetry} (restated below) follows as a corollary of \cref{prop:intermediate,prop:cyclic}.

\theoremasymmetry*

\begin{proof}
Observe that it is sufficient to prove that 
there is a prime \(p\) such that \(\nu_p(u_n)\gg n\)  
where the implied constant depends only on \(fg\) and \(p\).
More specifically, should such a linear growth estimate hold for the sequence of \(p\)-adic valuations \(\nu_p(u_n)\), then the desired linear growth estimate on the Weil height holds by \cref{prop:heights}.

Suppose that the hypergeometric sequence \(\seq[\infty]{u_n}{n=0}\in\mathscr{D}\) and that each of the irreducible factors of the associated polynomial \(fg\) has degree at most two. 
Since \(\seq[\infty]{u_n}{n=0}\in\mathscr{D}\), there is no possible rearrangement of the sequence's parameters for which \eqref{eq:large:same-roots} holds. 
Thus, by \cref{prop:intermediate}, we conclude that \(\seq[\infty]{u_n}{n=0}\) is \(p\)-asymmetric. In fact, the set of primes for which \(\seq[\infty]{u_n}{n=0}\) is \(p\)-asymmetric has positive density by \cref{prop:large:Galois}. 
The desired estimate for \(h_{\weil}(u_n)\) quickly follows from the effective growth bounds in \cref{lem:pasymmetric-divergence} for \(p\)-asymmetric hypergeometric sequences.

\emph{Mutatis mutandis,} the argument for hypergeometric sequences \(\seq[\infty]{u_n}{n=0}\in\mathscr{D}\) where the splitting field of the associated polynomials \(fg\) are cyclotomic is identical to that given in the previous case. (The only difference being that the \(p\)-asymmetry of \(\seq[\infty]{u_n}{n=0}\) follows from \cref{prop:cyclic}.)
\end{proof}

\section{The Membership Problem for hypergeometric sequences} 
\label{sec:membership}

\subsection*{Background and motivation}
In this section, we consider an application of our effective results in \cref{sec:divergence}.
The Membership Problem is an open decision problem concerning recursively defined sequences: %
Membership asks to procedurally determine whether a chosen target value is an element of a given sequence.
Perhaps the most well-known variant of Membership is the Skolem Problem.
The Skolem Problem asks to determine whether a given C-finite sequence vanishes (i.e., attains the value zero) at some index~\cite{everest2003recurrence}. (Here by a \emph{C-finite sequence} we mean a sequence that satisfies a linear recurrence relation with constant coefficients~\cite{everest2003recurrence, kauers2011tetrahedron}.)
Work in the 1980s established the decidability of Skolem for recurrences of order at most four~\cite{mignotte1984distance, vereshchagin1985occurence}; however, decidability at higher orders remains open.

Motivation for settling decidability of Skolem arises naturally in both theoretical computer science and pure mathematics.
Indeed, a proof that affirms the decidability of Skolem would be equivalent to a constructive proof of the Skolem--Mahler--Lech Theorem~\cite{everest2003recurrence}, which states that the set of indices \(\{n\in\N \mid  u_n = 0\}\) where a C-finite sequence \(\seq[\infty]{u_n}{n=0}\) vanishes is given by the union of a finite %
set and a finite number %
of infinite arithmetic progressions.

Given the simplicity of the model, it is, perhaps, surprising that decidability of the Membership Problem is open for hypergeometric sequences.
We take the opportunity to briefly sketch the obstacle to settling decidability in this setting (similar sketches are also given in \cite{kenison2024threshold, kenison2023membership, nosan2022membership}).
Given a recurrence relation of the form \eqref{eq:setup:def-u}, initial value \(u_0\in \QQ\), and rational target \(t\in\QQ\), the Membership Problem asks to determine whether there exists an \(n\in\NN\) such that \(u_n=t\).
From the product formulation for hypergeometric sequences (see \eqref{eq:setup:def-u-prod}), there is a straightforward argument that shows that the problem of deciding Membership in this setting reduces to that of deciding Membership for the subclass of hypergeometric sequences that either
diverge to infinity, or converges to some finite non-zero limit.
For sequences that do not converge to \(t\), we can compute a bound \(B\in\NN\) such that the terms in the tail subsequence \(\seq[\infty]{u_n}{n=B}\) all satisfy \(u_n \neq t\). 
Membership for such sequences then reduces to a finite search problem; that is to say, to determine Membership it suffices to determine whether \(t\in\{u_n \mid 0 \le n \le B-1\}\).
In the second case, where a sequence converges to \(t\), we can, without loss of generality, additionally assume that the sequence is eventually strictly monotonic.
Akin to the first case, we can compute a bound above which the terms in the tail subsequence do not equal the target \(t\) and so once again, deciding Membership reduces to a finite search problem.
Unfortunately, there is no known algorithm that can decide whether a hypergeometric sequence converges to a given rational limit.
This phenomenon is related to open conjectures on the nature of the gamma function (cf.~\cite{kenison2024threshold, nosan2022membership}).

For hypergeometric sequences, the Membership Problem is relatively straightforward to solve in the case where we can find a large prime $p$ such that $f$ and $g$ have different numbers of roots modulo $p$. 
Indeed, this is the methodology employed in \cite{kenison2023membership} to achieve the following.
\begin{lemma}[{\cite[Lemma 9]{kenison2023membership}}]
    \label{lem:pasymmetric-membership}
        The Membership Problem is decidable for the class of  hypergeometric sequences that are \(p\)-asymmetric for a given prime \(p\).
\end{lemma}
Briefly, the argument underpinning \cref{lem:pasymmetric-membership} works as follows. For a given \(p\)-asymmetric sequence, we can compute an effective bound $n_0$ such that for all $n \geq n_0$, we have $|{\nu_p(u_n)}| > |\nu_p(t)|$ (we can use the effective bounds in \cref{lem:pasymmetric-divergence}).
Thus, for all sufficiently large \(n\), we have that $u_n \neq t$ and so deciding Membership in this setting reduces to an exhaustive finite search problem.

\subsection*{Decidability results for Membership}
Recall \cref{thm:asymmetry} (\cref{sec:divergence}) that establishes effective bounds on the divergence of \(|\nu_p(u_n)|\) for certain hypergeometric sequences.
It is noteworthy that decidability results for Membership in this setting follow immediately from \cref{thm:asymmetry}.
 \begin{restatable}{corollary}{theoremintermediate}
 \label{thm:intermediate}
 The Membership Problem is decidable for hypergeometric sequences \(\seq[\infty]{u_n}{n=0}\) in \hyperlink{defin:D}{class \(\mathscr{D}\)} that, in addition, satisfy either of the following conditions.
\begin{enumerate}
    \item  Each of the irreducible factors of the polynomial \(fg\) has degree at most two.
    \item The splitting field of \(fg\) is cyclotomic.
\end{enumerate}
 \end{restatable}

\begin{proof}

Let \(\seq[\infty]{u_n}{n=0}\in\mathscr{D}\) be a hypergeometric sequence such that each of the irreducible factors of the associated polynomial \(fg\) has degree at most two. 
Since \(\seq[\infty]{u_n}{n=0}\in\mathscr{D}\), there is no possible rearrangement of the sequence's parameters for which \eqref{eq:large:same-roots} holds. Thus, by \cref{prop:intermediate}, we conclude that \(\seq[\infty]{u_n}{n=0}\) is \(p\)-asymmetric for some prime $p$. %
We note, by \cref{lem:pasymmetric-membership},  that the Membership Problem is decidable for \(p\)-asymmetric sequences.

\emph{Mutatis mutandis,} the argument that settles decidability of the Membership Problem for hypergeometric sequences \(\seq[\infty]{u_n}{n=0}\in\mathscr{D}\) where the splitting field of the associated polynomials \(fg\) are cyclotomic is identical to that given in the previous case. (The only difference being that the \(p\)-asymmetry of \(\seq[\infty]{u_n}{n=0}\) follows from \cref{prop:cyclic}.)
\end{proof}

Given the above application of effective divergence results, we take the opportunity to state a direction for future research.
On the one hand, the equidistribution results in the literature concerning quadratic congruences to prime moduli are not effective.
On the other hand, it appears plausible that an effective version of \cref{thm:toth-quant} (i.e., a statement of \cref{thm:toth-quant} where the constants $N_0$ and $c$ are effectively computable), can be obtained.
\begin{conjecture}
    The Membership Problem is decidable for hypergeometric sequences in class~\(\mathscr{C}\).
\end{conjecture}

\section{Acknowledgements}
\paragraph{Funding.} 
Florian Luca acknowledges funding from the 2024 ERC Synergy Grant DynAMiCs. Jakub Konieczny, Andrew Scoones, and James Worrell were supported by the EPSRC Fellowship EP/X033813/1.
Mahsa Shirmohammadi acknowledges funding from the ANR grant VeSyAM (ANR-22-CE48-0005).


\printbibliography

\end{document}